\documentclass[12pt,reqno]{amsart}
\usepackage{xifthen}
\usepackage{tabularx}
\usepackage{adjustbox}
\usepackage{makecell}
\usepackage{bbm}
\newcommand\scalemath[2]{\scalebox{#1}{\mbox{\ensuremath{\displaystyle #2}}}}

\newcommand{\lquotient}[2]{{\left.\raisebox{-.25em}{$#1$}\middle\backslash\raisebox{.25em}{$#2$}\right.}}
\newtheorem*{remark*}{Remark}

\usepackage{subfig}
\usepackage{pkgfile}

\begin{document}

	\title[Rational points in regular orbits] {Rational points in regular orbits attached to infinitesimal symmetric spaces}

\author{Trung Can}
\address{Department of Mathematics\\
Duke University\\
Durham, NC 27708}
\email{trung.can@duke.edu}

\author{Chung-Ru Lee}
\address{Department of Mathematics\\
Duke University\\
Durham, NC 27708}
\email{chung.ru.lee@duke.edu}

\author{Benjamin Nativi}
\address{Department of Mathematics\\
Duke University\\
Durham, NC 27708}
\email{benjamin.nativi@duke.edu}

\author{Gary Zhou}
\address{Department of Mathematics\\
Duke University\\
Durham, NC 27708}
\email{gaiting.zhou@duke.edu}
    
\thanks{Research supported by the DOmath program at Duke University.}

\begin{abstract}
Motivated by problems arising in the relative trace formula and arithmetic invariant theory, we prove the existence of rational points on orbits in certain infinitesimal symmetric spaces. As an application, we prove analogous results for orbits in certain global reductive symmetric spaces.
\end{abstract}
\date{July 17, 2018}
\maketitle


\section{Introduction}

Let $F$ be a characteristic zero field with fixed algebraic closure $\overline{F}$ and let $G$ be a connected reductive group over $F$ equipped with an involution (that is, an automorphism of order $2$)
$$
\theta:G \longrightarrow G.
$$
The Lie algebra $\mathfrak{g}$ of $G$ is then equipped with an induced automorphism, which we will denote as (by abuse of notation)
$$
\theta:\mathfrak{g} \longrightarrow \mathfrak{g},
$$
again of order $2$. We let $G(1)\subseteq G$ be the subgroup invariant under $\theta$, and let $\mathfrak{g}(\pm 1)$ be the $\pm 1$-eigenspaces of $\theta$ acting on $\mathfrak{g}$, respectively. Thus the Lie algebra of $G(1)$ is $\mathfrak{g}(1)$, and the adjoint action of $G(1)$ on $\mathfrak{g}$ preserves $\mathfrak{g}(-1)$. One can view $\mathfrak{g}(-1)$ as an \textbf{infinitesimal symmetric space}, since it can be identified with the tangent space of the reductive symmetric space $G/G(1)^{\circ}$ at the basepoint $G(1)^{\circ}$ (here the superscrupt $\circ$ is to denote the neutral component).  In particular the action of $G(1)$ on $\mathfrak{g}(-1)$ can be thought of as an infinitesimal analogue of the action of $G(1)$ on $G/G(1)^{\circ}$ by left multiplication.

The representations 
$$
G(1) \longrightarrow \mathrm{Aut}(\mathfrak{g}(-1)) 
$$
as $G$ and $\theta$ vary, appear in many contexts. The structure of the $G(1)$-orbits on $\mathfrak{g}(-1)$ is of particular interest and will be the main focus of this paper.

An element $X \in \mathfrak{g}(-1)(\overline{F})$ is said to be \textbf{relatively regular} if the $G(1)({\overline{F}})$-orbit of $X$ is of maximal dimension among all $G(1)({\overline{F}})$-orbits.  It is \textbf{regular} if the $G({\overline{F}})$-orbit is of maximal dimension among all $G({\overline{F}})$-orbits. We let 
$$
\mathfrak{g}(-1)^{rr} \subset \mathfrak{g}(-1)
$$
be the subscheme of relatively regular elements (see \S\ref{sec:orbits}). It is dense and $G(1)$-invariant. 

For $X \in \mathfrak{g}(-1)^{rr}(\overline{F})$, let
$$
\mathcal{O}(X)(\overline{F}):=G(1)(\overline{F})\cdot X.
$$
This is the set-theoretic orbit of $X$.  If $\mathcal{O}(X)(\bar{F})$ is invariant under $\mathrm{Gal}\,(\overline{F}/F)$ then it defines a subscheme over $F$,
$$
\mathcal{O} \subset \mathfrak{g}(-1)^{rr}.
$$
We refer to these subschemes as \textbf{relatively regular orbits} in $\mathfrak{g}(-1)$.

The $F$-points $\mathcal{O}(F)$ of $\mathcal{O}$ can then be concretely described as the set
$$
\mathcal{O}(F)=\mathcal{O}(X)(\bar{F}) \cap \mathfrak{g}(-1)^{rr}(F).
$$
Even though $\mathcal{O}(X)(\bar{F})$ contains $X$ and is therefore nonempty, it is false in general that $\mathcal{O}(F)$ is nonempty. This motivates the following question:
\begin{align}
\textit{Under which circumstances does it happen that $\mathcal{O}(F)$ is non-empty?}\label{question:1}
\end{align}

This question is of intrinsic number theoretic interest, and is moreover a problem that occurs often in representation theory and arithmetic invariant theory. For example, if $G$ is a connected reductive group that is quasi-split with simply connected derived group over $F$ then any conjugacy class in $G$ intersects $G(F)$. This is formulated in an important result of Kottwitz \cite{Kottwitz} that completed a work of Steinberg \cite{Steinberg}. The result is crucial for the stabilization of the trace formula \cite{KottwitzShelsted,Labesse}. We expect that the work we have begun in this paper will play an analogous role in the relative trace formula. We refer to \cite{Getz:Wambach} for an example of the type of comparison of relative trace formulae that results like this would apply.

Moreover the entirety of the subject of arithmetic invariant theory in the sense of Bhargava is based on a study of $F$-rational points of certain orbits like those discussed above; the thesis of Thorne is an excellent resource to consult for this point of view \cite{Thorne}.

In this paper, we answer (\ref{question:1}) in the following situation.

Let $p \geq q$ be positive integers.  Define the matrices
\begin{align*}
J_{p,q}=\left(\scalemath{0.7}{\begin{array}{*{10}{c}} J_p & \\ & -J_q\end{array}} \right)\qquad\text{and}\qquad J_{p,q}'=\left(\scalemath{0.7}{\begin{array}{*{10}{c}} J_p' & \\ & J_q'\end{array}} \right),
\end{align*}
where 
\begin{align*}
J_r:=\left(\scalemath{0.7}{\begin{array}{*{10}{c}} & & 1 \\ & \reflectbox{$\ddots$}  & \\ 1 & &\end{array}} \right)\qquad\text{and}\qquad J_r':=\left(\scalemath{0.7}{\begin{array}{*{10}{c}} & & 1 \\ & \reflectbox{$\ddots$}  & \\ (-1)^{r-1} & &\end{array}} \right).
\end{align*}
For any symmetric (resp.~skew-symmetric) matrix $J\in\mathrm{GL}_n(F)$ (resp.~$J'\in\mathrm{GL}_{2n}(F)$) we let
\begin{align*}
O(J)(R)& \defeq\{g \in\mathrm{GL}_n(R): Jg^{-t}J^{-1}=g\}\\
\mathrm{Sp}(J')(R)&\defeq\{g \in\mathrm{GL}_{2n}(R): J'g^{-t}J'^{-1}=g\}
\end{align*}
be the associated orthogonal group and symplectic group. In the orthogonal case we assume in addition that $|p-q| \leq 1$, and in the symplectic case we require that $p$ and $q$ are even.

We let $G$ be one of the groups among $\mathrm{GL}_{p+q}$, $O(J_{p,q})$, or $\mathrm{Sp}(J'_{p,q})$. Note that for any of these groups, the adjoint action by $I_{p,q}:=\left(\scalemath{0.6}{\begin{array}{*{10}{c}} \mathbbm{1}_p & \\ &-\mathbbm{1}_q \end{array}} \right)\in G$ induces an automorphism $\theta$ of $G$,
$$\theta(g)=I_{p,q}gI_{p,q}^{-1}.$$
The $\theta$-invariant subgroup $G(1)$ will be a direct product of two classical groups of same type as $G$.

The main result of this paper is as the following.
\begin{thm} \label{thm:main}
For $G$ and $\theta$ defined above, any $\mathrm{Gal}\,(\overline{F}/F)$-invariant regular $G(1)(\overline{F})$-orbit in $\mathfrak{g}(-1)(\overline{F})$ has an $F$-point.
\end{thm}

Previous studies include the case of $\mathrm{GL}_{p+q}$, where the result was obtained by Jacquet and Rallis in \cite{J-R}. In $\mathrm{O}_{p+q}$ it was obtained by J.~Thorne when $\abs{p-q}\leq 1$ \cite{Thorne}.

Our proof is essentially uniform in each of the above cases, and one expects that it can be generalized to a broader setting. The ultimate goal would be to prove analogues theorems to Kottwitz's Theorem for $\scalemath{0.8}{\lquotient{G(1)}{\mathfrak{g}(-1)}}$ under the setting of a general class in reductive groups. As a corollary of our Main Theorem, we prove consequential results for the orbits of certain classical groups on $M_{p,q}$, the space of $p \times q$ matrices, which is stated as Corollary \ref{cor:main} and will be further explained in \S\ref{sec:cor}.

For each $G$ above there is a natural action of $G(1)$ on $\mathrm{M}_{p,q}$, written explicitly by
\begin{align*}
G(1)(R) \times \mathrm{M}_{p,q}(R)  & \rightarrow \mathrm{M}_{p,q}(R)\\
(\left(\scalemath{0.7}{\begin{array}{*{10}{c}} g_1 & \\ & g_2 \end{array}} \right), X) & \mapsto g_1Xg_2^{-1}.
\end{align*}
We say an orbit in $\mathrm{M}_{p,q}$ is \textbf{$G(1)$-regular} if it is of maximal dimension among all $G(1)$-orbits.  

In Chapter 5, as a Corollary of the Main Theorem, we will prove the following.
\begin{cor} \label{cor:main}
For $G$ and $\theta$ defined above, any $\mathrm{Gal}(\overline{F}/F)$-invariant regular $G(1)$-orbit $\mathcal{O}$ in $\mathrm{M}_{p,q}$ has an $F$-point.
\end{cor}

Let us outline the contents of this paper. In Section \ref{section 2}, we recall the notions of an $\mathfrak{sl}_2$-triple, and a Kostant-Weierstrass section. These concepts are used to reduce the proof of Theorem \ref{thm:main} to exhibiting the existence of relatively regular nilpotent elements in $\mathfrak{g}(-1)(F)$. This part of the argument does not rely on the fact that $G$ is one of the three families of groups isolated above, but is also applicable in other settings. In \S \ref{sec:orbits} we record the dimension of the regular orbits, relating it to the rank of it corresponding locally symmetric space. In \S \ref{sec:comp} we use the result from \S\ref{sec:orbits} to exhibit the existence of a relatively regular nilpotent elements in $\mathfrak{g}(-1)(F)$. As an application of our Main Theorem, in \S \ref{sec:cor} we prove Corollary \ref{cor:main}.


\section*{Acknowledgements}

The authors would like to thank Prof. J. Getz for suggesting this project, help with editing and oversight throughout the project. The authors also thank Prof. H. Hahn and Prof. L Ng for providing this research opportunity through DOmath program in Summer of 2017. We would also want to thank Prof. Y. Sakellaridis for his insightful comment toward the subject.

\section{The $\mathfrak{sl}_2$-triples}\label{section 2}

In this section we apply the work of Kostant and Rallis \cite{K-R} to reduce the proof of Theorem \ref{thm:main} to exhibiting the existence of a relatively regular nilpotent element in $\mathfrak{g}(-1)(F)$.
To make this precise let us recall the notion of an $\mathfrak{sl}_2$-triple.
\begin{defn}
An $\mathfrak{sl}_2$-triple in a Lie algebra $\mathfrak{g}$ is a triple $(e,f,h)$ of non-zero elements in $\mathfrak{g}$ satisfying
$$[h,e]=2e\qquad[h,f]=-2f\qquad[e,f]=h.$$
\end{defn}
\indent In particular, if $(e,f,h)$ is an $\mathfrak{sl}_2$-triple 
then its $F$-span is naturally isomorphic to $\mathfrak{sl}_2$ as a Lie algebra; this explains the terminology.  

\begin{remark}
Any quotient in this paper is assumed to be a \textbf{GIT-quotient} without further specification.
\end{remark}

For the moment, we will assume $F=\mathbb{Q}$. Suppose we had fixed an involution $\theta$ on $\mathfrak{g}$ and defined $G(1)$ and $\mathfrak{g}(-1)$ accordingly. Let $\mathfrak{g}(-1)^e$ denote the set of centralizing elements for $e$ in $\mathfrak{g}(-1)$.
\begin{thm} \label{thm:KR} Assume the same notion as in the introduction.

Suppose that $(e,f,h)$ is an $\mathfrak{sl}_2$-triple in $\mathfrak{g}$ so that $e,f \in \mathfrak{g}(-1)^{rr}(\mathbb{Q})$ and $h \in \mathfrak{g}(1)(\mathbb{Q})$.  Then there exists a map
\begin{align*}
f+\mathfrak{g}(-1)^e \longrightarrow \lquotient{G(1)}{\mathfrak{g}(-1)}
\end{align*}
that is an isomorphism of schemes over $\mathbb{Q}$.
\end{thm}
\begin{remark}
This theorem is mostly due to Kostant and Rallis, but requires minor translation to bring to our setting. This is why we have restricted our attention to the $F=\mathbb{Q}$ situation. However, this is still strong enough to deduce our result for arbitrary $F$ (see Corollary 
\ref{cor:KR} below).
\end{remark}

\begin{proof}
By faithfully flat descent to verify that the given morphism is an isomorphism it suffices to check that 
\begin{align} \label{overC}
f+\mathfrak{g}^e_{\mathbb{C}} \longrightarrow \lquotient{G(1)_{\mathbb{C}}}{\mathfrak{g}(-1)_{\mathbb{C}}}
\end{align}
is an isomorphism.  

For this we note that in each case under consideration $\mathfrak{g}_{\mathbb{C}}$ is a complex Lie algebra admitting a real form $\mathfrak{g}_1$ such that $\theta$, defined as in the introduction, is the associated Cartan involution. More specifically, they correspond to types AIII, BDI and CII in \cite[Table V, p. 518]{Helgason}.  Thus we can apply 
 \cite[Theorems 8, 11, 12, and 13]{K-R} to deduce the theorem.
 \end{proof}

\begin{remark}
Technically, the Theorem in \cite{K-R} is stated for the adjoint group $G^\mathrm{ad}$ instead of $G$. However the adjoint representation
$$G\longrightarrow\mathrm{Aut}(\mathfrak{g})$$
factors through the adjoint group of $\mathfrak{g}$, which is $G^\mathrm{ad}=G/Z(G)$. Thus we have the isomorphism
$$\lquotient{G(1)^\mathrm{ad}_{\mathbb{C}}}{\mathfrak{g}(-1)_{\mathbb{C}}}
\simeq \lquotient{G(1)_{\mathbb{C}}}{\mathfrak{g}(-1)_{\mathbb{C}}}.
$$
\end{remark}

\begin{cor} \label{cor:KR}
Let $F=\mathbb{Q}$ and suppose that there is an $\mathfrak{sl}_2$-triple in $\mathfrak{g}$ with 
$$
e,f \in \mathfrak{g}(-1)^{rr}(\mathbb{Q}) \textrm{ and }h \in \mathfrak{g}(1)(\mathbb{Q}).
$$
Then for any characteristic zero field $k$,
every $G(1)(\bar{k})$-orbit in $\mathfrak{g}(-1)^{rr}(\overline{k})$ that is fixed under $\mathrm{Gal}(\overline{k}/k)$ intersects $\mathfrak{g}(-1)(k)$. 
\end{cor}

\begin{proof} Let $\mathcal{O}(\overline{k})$ be the $G(1)(\overline{k})$-orbit fixed by $\mathrm{Gal}(\overline{k}/k)$.  Its image $Y$ in $\scalemath{0.8}{\lquotient{G(1)}{\mathfrak{g}(-1)}}(\overline{k})$ is in $\scalemath{0.8}{\lquotient{G(1)}{\mathfrak{g}(-1)}}(k)$.

Denote by
$$
q:\mathfrak{g}(-1)(k) \longrightarrow \lquotient{G(1)}{\mathfrak{g}(-1)}(k)
$$
the quotient map. The inverse image $q^{-1}(Y)$ is the $k$-points of the Zariski closure of $\mathcal{O}(\overline{k})$. Regard $e$ and $f$ as elements in $\mathfrak{g}(k)$. We further know that $q^{-1}(Y) \cap \big(f+{\mathfrak{g}(-1)}^e(k)\big)$ is a single point which is contained in $q^{-1}(Y) \cap \mathfrak{g}(-1)^{rr}(k)$.  On the other hand, by \cite[Theorem 9]{K-R}, 
$$
q^{-1}(Y) \cap \mathfrak{g}(-1)^{rr}(k)=\mathcal{O}(\overline{k}) \cap \mathfrak{g}(-1)^{rr}(k).
$$
\end{proof}

In light of the corollary, to prove our main result it suffices to exhibit an $\mathfrak{sl}_2$-triple $(e,f,h)$ as in the assumptions of Corollary \ref{cor:main}.

\begin{lem} \label{lem:enough} Let $e \in \mathfrak{g}(-1)(F)$ be a relatively regular nilpotent element.  Then there exists an $\mathfrak{sl}_2$-triple $(e,f,h)$ with $f \in \mathfrak{g}(-1)^{rr}(F)$ and $h \in \mathfrak{g}(1)(F)$.
\end{lem}

\begin{proof}  The proof of \cite[Lemma 2.15]{Thorne} goes through without change in our context, but we fill in some details for the convenience of the reader.  
Let $e \in \mathfrak{g}(-1)(F)$ be a relatively regular element.  By the Jacobson-Morosov theorem \cite[Theorem 3]{Jacobson} it is an element of an $\mathfrak{sl}_2$-triple $(e,f',h') \in \mathfrak{g}(F)$.  Decompose $h'=h_{1}+h_{-1}$ and $f'=f_1+f_{-1}$ into eigenvectors under $\theta$ (with $h_i,f_i \in \mathfrak{g}(i)(F)$). Then 
$$
2e=[h',e]=[h_1,e]+[h_{-1},e]
$$
Since $[h_{-1},e] \in \mathfrak{g}(1)(F)$ and $[h_1,e] \in \mathfrak{g}(-1)(F)$ we deduce that $[h_{-1},e]=0$ and $[h_1,e]=2e$.    Because $[e,f']=h'$ we have $[e,f_1]=h_{-1}$ and $[e,f_{-1}]=h_1$.  Thus $h:=h_1$ is in the image of $\mathrm{ad}_e$ and $[h,e]=2e$.  
This implies that the pair $(e,h)$ can be completed to an $\mathfrak{sl}_2$ triple $(e,f,h)$ with $f \in \mathfrak{g}(F)$ by \cite[Corollary 3.5]{Kostant}; moreover, $f$ is uniquely determined.
Since $h \in \mathfrak{g}(1)(F)$ and $e \in \mathfrak{g}(-1)(F)$ if we let  $f'$ be the component of $f$ in $\mathfrak{g}(-1)(F)$ then $(e,f',h)$ is an $\mathfrak{sl}_2$-triple, so by the uniqueness result mentioned earlier we have $f'=f$.  Since $e$ is relatively regular, it follows that $f$ is as well, and this completes the proof.
\end{proof}
\begin{remark*}
Strictly speaking, Kostant assumes that $\mathfrak{g}$ is semisimple and $F=\mathbb{C}$ in the paper, but this is not necessary for the argument to be valid. 
\end{remark*}

In view of the lemma and Corollary \ref{cor:KR} to prove Theorem \ref{thm:main} it suffices to show that for every $G$ and $\theta$ as in the statement of that theorem there exists a relative regular $e \in \mathfrak{g}(-1)(\mathbb{Q})$.  The latter sections of this paper construct this element $e$.

\section{Dimensions of regular orbits}\label{sec:orbits}
In this section we compute the dimension of a relatively regular orbit in the cases under consideration.  

A Cartan subspace $\mathfrak{t}(-1)\subseteq\mathfrak{g}(-1)$ is a maximal commutative subspace which consists of semisimple elements. Its dimension is referred to as the \textbf{rank} of $\theta$.

\begin{lem} \label{lem:independence} If $\mathfrak{t}(-1) \subseteq \mathfrak{g}(-1)$ is a Cartan subspace then so is $\mathfrak{t}(-1)_E \subseteq \mathfrak{g}(-1)_E$ for every field extension $E/F$.
\end{lem}

\begin{proof}
Let $T(-1) \subseteq G$ be the connected subgroup whose Lie algebra is $\mathfrak{t}(-1)$.  Then $T(-1)$ is a maximal $\theta$-split torus in $G$.  Its base change to $E$ is therefore a maximal $\theta$-split torus in $G_E$ by \cite[Lemma 11.1]{Helminck:Wang}, and it follows that $\mathfrak{t}(-1)_E$ is a Cartan subspace of $\mathfrak{g}(-1)_E$.\end{proof}

Assume that $G$ is one of the groups among $\mathrm{GL}_{p+q}$, $\mathrm{O}(J_{p,q})$ or $\mathrm{Sp}(J_{p,q}')$ and $\theta$ is the conjugation by $I_{p,q}$. In view of Lemma \ref{lem:independence} the rank of $\theta$ in each of these cases is equal to the rank of the associated locally symmetric space, which can be found in \cite[Table V, \S X.6]{Helgason}:
\begin{enumerate}
\item[$\mathbf{AIII}$] For $G=\mathrm{GL}_{p+q}$, $\mathrm{rank}\,\theta=\min(p,q)$.
\item[$\mathbf{BDI}$] For $G=\mathrm{O}(J_{p,q})$,  $\mathrm{rank}\,\theta=\min(p,q)$.
\item[$\mathbf{CI}$] For $G=\mathrm{Sp}(J'_{p,q})$,  $\mathrm{rank}\,\theta=\frac{1}{2}\min(p,q)$.
\end{enumerate}
In particular, note that for $G=\mathrm{O}(J_{p,q})$ with $|p-q| \leq 1$ one has
$$
\mathrm{rank}\,\theta=\mathrm{rank}\,G,
$$
which is equivalent to saying that $\theta$ is a \textbf{stable involution} in the sense of \cite{Thorne}.

The relationship between the rank and regularity is stated in the following lemma.  It is a combination of  \cite[Lemma 2 and Proposition 8]{K-R}.

\begin{lem} \label{lem:rank}
An element $X \in \mathfrak{g}(-1)(F)$ is relatively regular if and only if 
$$
\mathrm{rank} \, \theta=\dim_F \mathfrak{g}(-1)^X.
$$\qed
\end{lem}

\section{Computations in Symmetric Spaces} \label{sec:comp}

In view of Lemma \ref{lem:enough} to prove Theorem \ref{thm:main} it suffices to construct a relatively regular nilpotent element in $\mathfrak{g}(-1)(F)$ for the cases specified in the introduction.   This is the goal of the ongoing section.

Without loss of generality, assume $p\geq q$ from now on. For presentation convenience, we let $\epsilon_n\in\mathfrak{gl}_n$ be a regular nilpotent element of the form
$$ \epsilon_n=\left(\scalemath{0.7}{\begin{array}{*{10}{c}}
 & \mathbbm{1}_{n-1}\\
0 & \\
\end{array}} \right),
$$
let the element $\lambda_{m,n}\in\mathrm{M}_{m,n}$ be,
$$\lambda_{m,n}=\left(\scalemath{0.7}{\begin{array}{*{10}{c}}
 & \mathbf{0}_{m-1,n-1}\\
1 & \\
\end{array}} \right)$$
and let $K_n\in\mathfrak{gl}_n$ be the diagonal matrix 
$$K_n= \left(\scalemath{0.6}{\begin{array}{*{10}{c}} 
1 & && \\
& -1 & &\\
& & \ddots &\\
& & & (-1)^{n-1}
\end{array}} \right).$$

\subsection*{General Linear Groups} For $p \geq q$ and $n=p+q$, we let $G=\mathrm{GL}_n$, $\mathfrak{g}=\mathfrak{gl}_n$, and $\theta$ be as in the introduction. We write $\mathbf{0}_{m,n}$ for an $m\times n$ zero matrix and omit an index when $m=n$.

We will prove that
$$e = \left(\scalemath{0.9}{\begin{array}{ccc|c} 
& & & \mathbbm{1}_q \\
& & & \mathbf{0}_{p-q,q} \\
\hline
\mathbf{0}_{q,1} & \mathbbm{1}_q & \mathbf{0}_{q,p-q-1} &\\
\end{array}} \right)$$
is a relatively regular nilpotent element in $\mathfrak{g}(-1)$.
\begin{remark*}
In case of $p=q$, the nilpotent element is defined as $$e = \left(\scalemath{0.8}{\begin{array}{*{10}{c}} 
 &  \mathbbm{1}_q \\
\epsilon_q &  
\end{array}}  \right).$$
\end{remark*}
The nilpotency can be verified by direct computation. As mentioned in \S \ref{sec:orbits}, for $G=\mathrm{GL}_{p+q}$, we have
$$\mathrm{rank}\,\theta=q.$$
Thus, it suffices to prove the following
\begin{lem}
With the setting as above, we have
$$\dim\mathfrak{g}(-1)^e=q.$$
\end{lem}

\begin{proof}
Note that 
$$
\mathfrak{g}(-1)(F)=\left\{X=\left(\scalemath{0.7}{\begin{array}{*{10}{c}} & A \\ B &  \end{array}} \right)\Bigm|A \in \mathrm{M}_{p,q}(F) \textrm{ and }B \in \mathrm{M}_{q,p}(F)\right\}.
$$
The set $\mathfrak{g}(-1)^e$ is given by the identity
\begin{align}
\mathrm{ad}_eX & =0\label{eq:gl-nilpotent}
\end{align}
Hence, for any $X$ satisfying (\ref{eq:gl-nilpotent}), the respective $A$ and $B$ satisfy
\begin{align*}
\left(\scalemath{0.7}{\begin{array}{*{10}{c}} \mathbf{0}_{q,1} &\mathbbm{1}_q& \mathbf{0}_{q,p-q-1} \end{array}} \right) A & = B \left(\scalemath{0.7}{\begin{array}{*{10}{c}} \mathbbm{1}_q\\ \mathbf{0}_{p-1,q} \end{array}} \right),\text{ and}\\
\left(\scalemath{0.7}{\begin{array}{*{10}{c}} \mathbf{0}_{q,1}&\mathbbm{1}_q & \mathbf{0}_{q,p-q-1} \end{array}} \right) & = \left(\scalemath{0.7}{\begin{array}{*{10}{c}} \mathbbm{1}_q\\ \mathbf{0}_{p-q,q} \end{array}} \right) B.
\end{align*}
Further, we will denote $A_r$ to be the $r$-th row of A and $B_r$ as the $r$-th column of $B$. We then have
\begin{align*}
(\scalemath{0.7}{\begin{array}{*{10}{c}} \mathbf{0}_{q,1}&\mathbbm{1}_q&\mathbf{0}_{q,p-q-1} \end{array}} ) A & = \left(\scalemath{0.7}{\begin{array}{*{10}{c}} A_2\\ \vdots \\ A_{q+1} \end{array}} \right)\\
B (\scalemath{0.7}{\begin{array}{*{10}{c}} \mathbbm{1}_q\\ \mathbf{0}_{p-q,q} \end{array}} ) & = (\scalemath{0.6}{\begin{array}{*{10}{c}} B_1 & \cdots & B_q \end{array}} ).
\end{align*}
Thus, substituting we have $b_{i,j} = a_{i+1,j}$ for $1 \leq i,j \leq q$. Observe also
\begin{align*}
A (\scalemath{0.7}{\begin{array}{*{10}{c}} \mathbf{0}_{q,1}&\mathbbm{1}_q&\mathbf{0}_{q,p-q-1} \end{array}} ) & = (\scalemath{0.7}{\begin{array}{*{10}{c}} \mathbf{0}_{p,1} &A&\mathbf{0}_{p,p-q-1} \end{array}} )\\
(\scalemath{0.7}{\begin{array}{*{10}{c}} \mathbbm{1}_q\\ \mathbf{0}_{p-q,q} \end{array}} ) B & = (\scalemath{0.7}{\begin{array}{*{10}{c}} B\\ \mathbf{0}_{p-q,p} \end{array}} ).
\end{align*}
This implies that $b_{i,j} = 0$ for either $j=1$ or $q +1 < j \leq p$, and $a_{i,j} = 0$ for $q < i \leq p$. Otherwise, $b_{i,j} = a_{i,j-1}$.

Using the results we also conclude that $a_{i,j} = a_{i+1,j+1}$. The matrices $A$ and $B$ therefore must be of the form:\\
$$A = \left(\substack{\scalemath{0.6}{\begin{array}{*{10}{c}}
a_1 & a_2 & \cdots & a_q \\
~ & a_1 & \ddots & \vdots \\ 
~ & ~ & \ddots & a_2\\
~ & ~ & ~ & a_1\end{array}}\\
\centering\mathbf{0}_{p-q,q}
} \right)\qquad\text{and}\qquad
B = \left(\substack{\scalemath{0.7}{\begin{array}{*{10}{c}}
0&a_1 & a_2 & \cdots & a_q \\
~& 0 & a_1 & \ddots & \vdots \\ 
~&~ & \ddots  & \ddots & a_2\\
~&~ & ~ & 0 & a_1\end{array}}\mathbf{0}_{q,p-q-1}} \right).$$\\

Thus, $\mathfrak{g}(-1)^e$ has dimension $q$. Notice that in the special case of $p=q$, the matrix $B$ is obtained by eliminating the $(q+1)$-th column.
\end{proof}
\begin{cor}
The element $e\in\mathfrak{gl}_{p+q}(-1)(F)$ is nilpotent and relatively regular.
\end{cor}

\subsection*{Symplectic Groups}
Let $J'_{p,q}$ as defined in the introduction. Consider the symplectic group $G=\mathrm{Sp}(J'_{p,q})$ and its adjoint representation on $\mathfrak{g}=\mathfrak{sp}(J'_{p,q})$. The involution $\theta$ is defined by the conjugation by $I_{p,q}$, and we will focus on the action of $G(1)$ on $\mathfrak{g}(-1)$ as usual.

Throughout, let $r = \frac{p-q}{2}$.
If $p\neq q$, let $$e = \left(\scalemath{0.9}{\begin{array}{ccc|c} 
&&& \mathbf{0}_{r-1,q}\\
&&& \mathbbm{1}_q \\
&&& \mathbf{0}_{r+1,q} \\
\hline
\mathbf{0}_{p,r+1} & (-1)^r\mathbbm{1}_q & \mathbf{0}_{p,r-1} &
\end{array}} \right).$$
and if $p=q$, let
$$e = \left(\scalemath{0.8}{\begin{array}{*{10}{c}} 
 &  \epsilon_q  \\
\epsilon_q &  
\end{array}}  \right).$$\\

As mentioned in Section \ref{sec:orbits},
$$\mathrm{rank}\,\theta=\frac{q}{2}.$$
Furthermore, we have the following lemma. 
\begin{lem}
With the settings above, we have
$$\dim\mathfrak{g}(-1)^e=\frac{q}{2}.$$
\end{lem}
\begin{proof}
We have, by definition,
$$\mathfrak{g}(-1)^e = \{ X \in \mathfrak{g}(-1) \mid  \left[X,e\right]=0\text{, and }\theta(X)=-X \}.$$

We will begin with a general $X=\scalemath{0.6}{
\begin{pmatrix}
& A\\
B &
\end{pmatrix}}$ with $A\in\mathrm{M}_{p,q}$ and $B\in\mathrm{M}_{q,p}$.

Observe that by $\left[X,e\right]=0$, we have
\begin{align}
A(\scalemath{0.8}{\begin{array}{*{10}{c}} \mathbf{0}_{q,r+1}&(-1)^r \mathbbm{1}_q&\mathbf{0}_{q,r-1} \end{array}} ) & = (\scalemath{0.8}{\begin{array}{*{10}{c}} \mathbf{0}_{p,r+1}&(-1)^r A&\mathbf{0}_{p,r-1} \end{array}} )\notag\\
& = \left(\scalemath{0.7}{\begin{array}{*{10}{c}} \mathbf{0}_{r-1,p}\\B\\\mathbf{0}_{r+1,p} \end{array}} \right) = \left(\scalemath{0.7}{\begin{array}{*{10}{c}} \mathbf{0}_{r-1,q}\\ \mathbbm{1}_q\\\mathbf{0}_{r+1,q} \end{array}} \right) B\label{eq:sp centralizer}
\end{align}
Thus, by comparison we have the relations
\begin{align*}
a_{i,j} & = 0\quad\text{  if  }\quad 1 \leq i \leq r-1 \ \text{ or }\  q+r \leq i \leq p\\
b_{i,j} & = 0\quad\text{  if  }\quad 1 \leq j \leq r+1 \ \text{ or }\  q+r+2 \leq j \leq p.  
\end{align*}
For any yet unspecified $1 \leq i,j \leq p$, we have $b_{i,j} = (-1)^r a_{i,j}$. This determines $\mathfrak{g}(-1)^e$ to be of the form $$\left( \scalemath{0.8}{\begin{array}{ccc|c} 
&&& \mathbf{0}_{r-1,q}\\
&&& A_q \\
&&& \mathbf{0}_{r+1,q}\\
\hline
\mathbf{0}_{q,r+1} & (-1)^r A_q & \mathbf{0}_{q,r-1} &
\end{array}} \right),$$
where $A_q$ is a $q\times q$ matrix and we denote its entries by $x_{i,j}$ for $1 \leq i,j \leq q$.

Also, observe that (\ref{eq:sp centralizer}) also provides the relation $(-1)^r A_q N'_q = (-1)^r N'_q A_q$ for the $q\times q$ matrix $A_q$, where
$$N'_q = \left(\scalemath{0.6}{\begin{array}{*{10}{c}}
& \mathbbm{1}_{q-2}\\
\mathbf{0}_{2} &
\end{array}} \right).$$
Written explicitly, it spells $x_{i,j} = x_{i+2,j+2}$. Also, we let $x_{i,j} = 0$ if $i > q-2$ and $j \leq q-2$ or if $j > q-2$ and $i \leq q-2$.

Summarizing, we have $$ A_q = 
\left(\scalemath{0.7}{\begin{array}{*{10}{c}} 
x_1 & x_2 & \cdots & \cdots & x_q\\
x_0 & y_{1}& y_{2} &\ddots&\vdots\\ 
0 &y_{0}&\ddots&\ddots&\vdots\\
\vdots & \ddots &\ddots&x_1& x_2\\
0 & \cdots & 0 & x_0 & y_{1}\\
\end{array}} \right).$$

Finally, we apply the constraint that the centralizer lies in $\mathfrak{g}=\mathfrak{sp}(J'_{p,q})$, that is,
$$J'_{p,q}XJ'_{p,q} = X^t.$$
From this equation we know that when $i+j$ is even, we have $a_{q+1-i,q+1-j} = a_{j,i}$. Meanwhile, if $i+j$ is odd, we have $a_{q+1-i,q+1-j} = -a_{j,i}$. This determines the explicit form of $A_q$,

$$A_q = \left(\scalemath{0.6}{\begin{array}{*{10}{c}}
x_1&0&x_3&\cdots & 0 \\
~&x_1&\ddots &\ddots &\vdots\\
~&~&\ddots& \ddots &x_3\\
~&~&~&\ddots &0\\
~&~&~&~&x_1 \end{array}} \right).$$

It follows that the centralizing elements $g(-1)^e$ is a vector space of dimension $\frac{q}{2}$.
\end{proof}
\begin{cor}
The element $e\in\mathfrak{sp}(J'_{p,q})(-1)(F)$, constructed above (for the respective $p,q$), is nilpotent and relatively regular.
\end{cor}
\subsection*{Orthogonal Group} For $|p-q|\leq 1$, we let $G=\mathrm{O}(J_{p,q})$, and $\theta$ be as in the introduction. Throughout, without the loss of generality, assume $p \ge q$. As mentioned above, in this case $\mathrm{rank}\,\theta=\mathrm{rank}\, G = q$, so the involution $\theta$ is stable (in the sense of \cite{Thorne}).

Again, we construct explicit formulae for $e$, and demonstrate by computation that the centralizer of $e$ in $\mathfrak{g}(-1)$ has dimension $q$.

\noindent\textit{Case 1}

When $p=q+1$, our element is given by $$e = \left(\scalemath{0.7}{\begin{array}{cc|c} 
&  & \mathbbm{1}_q \\
&  & \mathbf{0}_{1,q} \\
\hline
\mathbf{0}_{q,1} & \mathbbm{1}_q
\end{array}}  \right).$$
\begin{lem} Under this setting,
$$\dim\mathfrak{g}(-1)^e=q.$$
\end{lem}
\begin{proof} 
A centralizing element $X=\scalemath{0.6}{
\begin{pmatrix}
& A\\
B &
\end{pmatrix}}$ with $A\in\mathrm{M}_{p,q}$ and $B\in\mathrm{M}_{q,p}$ of $e$ in $\mathfrak{g}(-1)^e$ satisfies
\begin{enumerate}
\item $\mathrm{ad}_eX=0$, and
\item $B=J_qA^tJ_{q+1}$.
\end{enumerate}
The first constrain is the same as worked in the general linear case. The matrices $A$ and $B$ must be of the form
$$A = \left(\scalemath{0.6}{\begin{array}{*{10}{c}}
a_1 & a_2 & \cdots & a_q \\
0 & a_1 & \ddots & \vdots \\ 
\vdots & 0 & \ddots & a_2\\
\vdots & \vdots & \ddots & a_1\\
0 & 0 &  \cdots & 0 \end{array}} \right),\quad
B = \left(\scalemath{0.6}{\begin{array}{*{10}{c}}
0&a_1 & a_2 & \cdots & a_q  \\
0 & 0 & a_1 & \ddots & \vdots  \\ 
\vdots & \vdots & \ddots  & \ddots & a_2\\
0 &0  & \cdots & 0 & a_1  \end{array}} \right).$$

Indeed, $B= J_q A^T J_{q+1}$ and thus
$$\dim\mathfrak{g}(-1)^e=q$$
is valid.
\end{proof}

\noindent\textit{Case 2}

When $p=q$, the (to be verified) nilpotent element is
$$e= \begin{pmatrix}
~& e^*\\
J_q (e^*)^t J_q & ~\\
\end{pmatrix},$$ 

where $e^*$ is obtained by shifting the first $\ceil{q/2}$ entries of the identity $I_q$ to the right by one column (see below for the explicit form). For convenience, $J$ will refer to $J_q$ in the following computation.

If $p=q=2k$ is even, then
$$e^* = \left(\scalemath{0.5}{\begin{array}{*{10}{c}}
0 & 1 & ~  \\
\vdots & & \ddots \\
0 & & & 1 \\
0 & & & 1 & \\
\vdots &  & & & \ddots \\  
0 &  & & & & 1 
\end{array}}  \right) =\begin{pmatrix}
\epsilon_k&\lambda_{k,k}\\
\mathbf{0}_{k}&\mathbbm{1}_k\end{pmatrix}.$$
where $\epsilon_k = \left(\scalemath{0.7}{\begin{array}{*{10}{c}}
 & I_{k-1}\\
0 & \\
\end{array}} \right)$ and $\lambda_{k,k} = \left(\scalemath{0.7}{\begin{array}{*{10}{c}}
 & \textbf{0}_{k-1}\\
1 & \\
\end{array}} \right)$, as defined before.
\begin{lem} Under this setting, we have
$$\dim\mathfrak{g}(-1)^e.$$
\end{lem}
\begin{proof} The centralizing elements are of the form $\left(\scalemath{0.7}{\begin{array}{*{10}{c}}
~&X\\
JX^tJ&~
\end{array}} \right)$. For detailed computation, write $X$ in the block form $X=\left(\scalemath{0.6}{\begin{array}{*{10}{c}}
A & B\\
C & D
\end{array}} \right)$ with $A,B,C,D \in \mathrm{M}_{k}(F)$. 

From the relation
$$eX=Xe$$
we obtain 
$$XJ(e^*)^t=e^*JX^t \qquad\textrm{ and }\qquad X^tJe^*=(e^*)^tJX.$$
Expanding blockwise yields $$\left(\scalemath{0.7}{\begin{array}{*{10}{c}}
\epsilon_k&\lambda_k\\
\mathbf{0}_k&I_k
\end{array}} \right)J\left(\scalemath{0.7}{\begin{array}{*{10}{c}}
A^t&C^t\\
B^t&D^t\\
\end{array}} \right)=\left(\scalemath{0.7}{\begin{array}{*{10}{c}}
A&B\\
C&D\\
\end{array}} \right) J\left(\scalemath{0.7}{\begin{array}{*{10}{c}}
\epsilon_k^t&\mathbf{0}_k\\
\lambda_k^t&I_k\\
\end{array}} \right)$$ and $$\left(\scalemath{0.7}{\begin{array}{*{10}{c}}
\epsilon_k^t&\mathbf{0}_k\\
\lambda_k^t&I_k\\
\end{array}} \right) J \left(\scalemath{0.7}{\begin{array}{*{10}{c}}
A&B\\
C&D\\
\end{array}} \right)=\left(\scalemath{0.7}{\begin{array}{*{10}{c}}
A^t&C^t\\
B^t&D^t\\
\end{array}} \right)J\left(\scalemath{0.7}{\begin{array}{*{10}{c}}
\epsilon_k&\lambda_k\\
\mathbf{0}_k&I_k\\
\end{array}} \right).$$\\
Solving the equalities entrywise, we obtain \begin{align*}
A & = \left(\scalemath{0.7}{\begin{array}{*{10}{c}}
0&a_1&\cdots&a_{k-1}\\
~&0&\ddots&\vdots\\
~&~&\ddots&a_1\\
~&~&~&0\\
\end{array}}  \right) ,& B & =
B_1+B_2+B_3,\\
C &=\textbf{0}_k ,& D &= \left(\scalemath{0.7}{\begin{array}{*{10}{c}}
a_1 &\hdots & a_{k-1}&a_k-a_{2k}\\
~& a_1 &\ddots&a_{k-1}\\
~&~&\ddots&\vdots\\
~&~&~&a_1
\end{array}} \right).
\end{align*}
The matrices $B_i$, with $i=1,2,3$, are defined as
$$
B_1=\scalemath{0.7}{\begin{pmatrix}
a_k & a_{k+1} & \cdots & a_{2k-2} & a_{2k-1}\\
 & a_k & \ddots & \ddots & a_{2k-2}\\
 & & \ddots & \ddots & \vdots\\
 & & & a_k & a_{k+1}\\
 & & & & a_{2k}
\end{pmatrix}}\qquad
B_2=\scalemath{0.7}{\begin{pmatrix}
0 & & & & \\
a_{k-1} & \ddots & & & \\
\vdots & \ddots & \ddots & & \\
a_2 & a_3 & \ddots & \ddots &\\
a_1 & a_2 & \hdots & a_{k-1} & 0
\end{pmatrix}}\qquad
B_3=\scalemath{0.7}{\begin{pmatrix}
0 & & & & \\
a_{k-1} & \ddots & & & \\
\vdots & \ddots & \ddots & & \\
a_2 & \cdots & a_{k-1} & 0 &\\
0 & 0 & \hdots & 0 & 0
\end{pmatrix}}.
$$

Therefore, counting the number of independent variables, the centralizing elements $\mathfrak{g}(-1)^e$ forms a vector subspace of dimension $2k$.
\end{proof}

If $p=q=2k+1$ is odd, we define 

$$e^* = \left(\scalemath{0.5}{\begin{array}{*{10}{c}}
0 & 1 &  \\
\vdots & & \ddots \\
0 & & & 1 \\
0 & & & & 1 \\
0 & & & & 1 & \\
\vdots & & & & & \ddots \\  
0 &  & & & & & 1 
\end{array}} \right)=\begin{pmatrix}
\epsilon_{k+1}&\lambda_{k+1,k}\\
0_{k,k+1}&I_k\end{pmatrix}$$
which occurs in the definition of $e$.

The computation in this case is analogous to the previous case.
\begin{lem} Under this setting, we have
$$\dim\mathfrak{g}(-1)^e.$$
\end{lem} 
\begin{proof} Similar to the previous proof, assume that a centralizing element is of the form $\left(\scalemath{0.6}{\begin{array}{*{10}{c}}
~&X\\
JX^tJ&~
\end{array}} \right)$ where $X$ can be written as $\left(\scalemath{0.6}{\begin{array}{*{10}{c}}
A&B\\
C&D\\
\end{array}} \right)$ with $A \in \mathrm{M}_{k+1}(F)$, $B \in \mathrm{M}_{k+1,k}(F)$, $C \in \mathrm{M}_{k,k+1}(F)$, and $D \in \mathrm{M}_{k}(F)$.

The equality $\left(\scalemath{0.6}{\begin{array}{*{10}{c}}
&X\\
JX^tJ&\\
\end{array}} \right)
\left(\scalemath{0.6}{\begin{array}{*{10}{c}}
&e^*\\
J(e^*)^tJ&\\
\end{array}} \right)=
\left(\scalemath{0.6}{\begin{array}{*{10}{c}}
&e^*\\
J(e^*)^tJ&\\
\end{array}} \right)
\left(\scalemath{0.6}{\begin{array}{*{10}{c}}
&X\\
JX^tJ&\\
\end{array}} \right)$ reduces to
$$XJ(e^*)^t=e^*JX^t \quad\textrm{ and }\quad X^tJe^*=(e^*)^tJX.$$ 

Writing these equalities blockwise yields
\begin{align*}\left(\scalemath{0.7}{\begin{array}{*{10}{c}}
\epsilon_{k+1}&\lambda_{k+1,k}\\
\mathbf{0}_{k,k+1}&\mathbbm{1}_{k}
\end{array}} \right) J \left(\scalemath{0.7}{\begin{array}{*{10}{c}}
A^t&C^t\\
B^t&D^t\\
\end{array}} \right) & =\left(\scalemath{0.7}{\begin{array}{*{10}{c}}
A&B\\
C&D\\
\end{array}} \right) J \left(\scalemath{0.7}{\begin{array}{*{10}{c}}
\epsilon_{k+1}^t&\mathbf{0}_{k+1,k}\\
\lambda_{k,k+1}^t&\mathbbm{1}_{k}\\
\end{array}} \right)\text{, and}\\
\left(\scalemath{0.7}{\begin{array}{*{10}{c}}
\epsilon_{k+1}^t&\mathbf{0}_{k+1,k}\\
\lambda_{k,k+1}^t&\mathbbm{1}_{k}\\
\end{array}} \right) J \left(\scalemath{0.7}{\begin{array}{*{10}{c}}
A&B\\
C&D\\
\end{array}} \right) & =\left(\scalemath{0.7}{\begin{array}{*{10}{c}}
A^t&C^t\\
B^t&D^t\\
\end{array}} \right)J\left(\scalemath{0.7}{\begin{array}{*{10}{c}}
\epsilon_{k+1}&\lambda_{k+1,k}\\
\mathbf{0}_{k,k+1}&\mathbbm{1}_{k}
\end{array}} \right).
\end{align*}
Solving the equalities entrywise, we obtain \begin{align*}
A &= \left(\scalemath{0.7}{\begin{array}{*{10}{c}}
0&a_1&a_2&\hdots&a_k\\
~&0&a_1&\ddots&a_{k-1}\\
~&~&\ddots&\ddots &\vdots\\
~&~&~&0&a_1\\
~&~&~&~&0\\
\end{array}} \right),& B &=B_1+B_2+B_3,\\
C&=\textbf{0}_{k,k+1},& D &=\left(\scalemath{0.7}{\begin{array}{*{10}{c}}
a_1&\hdots&a_{k-1}&a_k\\
~&a_1&\ddots&a_{k-1}\\
~&~&\ddots&\ddots\\
~&~&~&a_1\\
\end{array}} \right),
\end{align*}
whereas
$$
B_1=\scalemath{0.6}{\begin{pmatrix}
a_{k+1} & a_{k+2} & \cdots & a_{2k-1} & a_{2k}\\
0 & a_{k+1} & \ddots & \ddots & a_{2k-1}\\
 & \ddots & \ddots & \ddots & \vdots\\
 & & \ddots & a_{k+1} & a_{k+2}\\
 & & & 0 & a_{2k+1}\\
 & & & & -a_{2k+1}
\end{pmatrix}}\qquad
B_2=\scalemath{0.6}{\begin{pmatrix}
0 & & & & \\
a_k & \ddots & & & \\
\vdots & \ddots & \ddots & & \\
a_3 & a_4 & \ddots & \ddots &\\
a_2 & a_3 & \hdots & a_k & 0\\
a_1 & a_2 & \hdots & a_{k-1} & a_{k+1}
\end{pmatrix}}\qquad
B_3=\scalemath{0.6}{\begin{pmatrix}
0 & & & & \\
a_{k-1} & \ddots & & & \\
\vdots & \ddots & \ddots & & \\
a_3 & \cdots & a_{k-1} & 0 & 0\\
 & & \mathbf{0}_{2,k} & & 
\end{pmatrix}}.
$$
We observe that $\dim\mathfrak{g}(-1)^e=2k+1$ as claimed.
\end{proof}
\begin{cor}
The element $e\in\mathfrak{o}(J_{p,q})(-1)(F)$ constructed above  (for the respective $p,q$), is nilpotent and relatively regular.
\end{cor}

\section{Proof of Corollary \ref{cor:main}} \label{sec:cor}

We place ourselves in the situation of Corollary \ref{cor:main}.   Assume first that $G=\mathrm{GL}_{p+q}$.  In this case the rank of a matrix in $\mathrm{M}_{p,q}(\bar{F})$ is invariant under the $G(1)=\mathrm{GL}_p(\bar{F}) \times \mathrm{GL}_q(\bar{F})$-action, and all matrices of a given rank are in the same orbit.  Thus the corollary is trivial in this case.  We can therefore assume $G=\mathrm{O}(J_{p,q})$ or $G=\mathrm{Sp}(J_{p,q}')$, where in the orthogonal case we assume in addition that $\abs{p-q}\leq 1$.

One has a linear isomorphism $\mathfrak{g}(-1) \simeq \mathrm{M}_{p,q}$ given on points in an $F$-algebra $R$ by 
\begin{align*}
\mathfrak{g}(-1)(R) &\longrightarrow \mathrm{M}_{p,q}(R)\\
\left(\scalemath{0.7}{\begin{array}{*{10}{c}} & X\\
Y & \end{array}}  \right) &\longmapsto X
\end{align*}
Note that for any element in $\mathfrak{g}(-1)(R)$, $Y$ is determined by its counterpart $X$ via the relation
$$Y=J_qX^tJ_p$$
when $G=\mathrm{O}(J_{p,q})$; and
$$Y=-J'_qX^tJ'_p$$
when $G=\mathrm{Sp}(J'_{p,q})$. This proves the map is an isomorphism.
It is $G(1)$-equivariant with respect to the conjugation action on the left hand side and the the action explained before the statement of the corollary on the right hand side. Thus an orbit in $\mathrm{M}_{p,q}$ is regular if and only if its preimage is regular.  The corollary follows from Theorem \ref{thm:main}.
\qed

\bibliography{refs}{}
\bibliographystyle{alpha}

\end{document}